\documentclass[article]{amsart}
\usepackage{graphicx}
\usepackage{latexsym}
\usepackage{hyperref}
\usepackage{amsfonts}
\usepackage[all]{xy}
\usepackage{amssymb, mathrsfs, amsfonts, amsmath}
\usepackage{cancel}

\usepackage{enumerate}
\usepackage{graphicx}
\usepackage{eulervm}
\usepackage{geometry}

\newtheorem{theorem}{Theorem}

\newtheorem{corollary}{Corollary}

\newtheorem{definition}{Definition}

\newtheorem{lemma}{Lemma}

\newtheorem{remark}{Remark}

\newcommand{\diver}{{{\rm{div}\,}}}

\newcommand{\grad}{{{\rm grad\,}}}

\DeclareMathOperator{\Hess}{{{\rm Hess}\,}}
\DeclareMathOperator{\hess}{{{\rm hess}\,}}
\DeclareMathOperator{\tr}{Tr}
\DeclareMathOperator{\Ll}{L}
\DeclareMathOperator{\Pp}{P}

\begin{document}

\title{Maximum principle for  semi-elliptic trace operators and  geometric applications}
{\author{G. Pacelli Bessa }
\address{Departamento de Matem\'{a}tica\\Universidade Federal do Cear\'{a}-UFC\\
60455-760, Fortaleza-Brazil} \email{bessa@mat.ufc.br} \thanks{Research partially supported by CAPES-CNPq-Brazil}}
{\author{Leandro F. Pessoa}
\address{Departamento de Matem\'{a}tica\\Universidade Federal do Cear\'{a}-UFC\\
60455-760, Fortaleza-Brazil} \email{leandrofreitasp@yahoo.com.br} \thanks{Research partially supported  by CAPES}}
\date{\today}

\maketitle

\begin{abstract} Based on ideas of L. Al\'{\i}as, D. Impera and M. Rigoli developed in \cite{alias-impera-rigoli}, we  present a fairly general  weak/Omori-Yau maximum principle for trace operators. We apply this version of maximum principle to generalize several higher order mean curvature estimates and  to  give an extension of  Alias-Impera-Rigoli Slice Theorem of \cite[Thm. 16 \& 21]{alias-impera-rigoli}, see Theorems \ref{thm5}, \ref{thm6}.
\end{abstract}

\section{Introduction} The theory of minimal and constant mean curvature hypersurfaces of  product spaces $N\times \mathbb{R}$, where $N$ is a complete Riemannian manifold,
 has been developed into a  rich
theory  \cite{meeks-rosenberg1}, \cite{rosenberg}, \cite{rosenberg-book}   yielding a wealth of examples
 and  results, see for instance  \cite{abresch-Rosenberg}, \cite{alencar-doCarmo-Fernadez-Tribuzy}, \cite{alencar-doCarmo-Tribuzzy}, \cite{alencar-doCarmo-Tribuzzy2}, \cite{alias-bessa-dajczer-MathAnn}, \cite{alias-bessa-montenegro},  \cite{alias-dajczer-3}, \cite{bessa-costa}, \cite{bessa-montenegro2}, \cite{doCarmo-fernandez}, \cite{earp} \cite{elbert-rosenberg}, \cite{hauswirth}, \cite{hoffman-lira-Rosenberg},  \cite{nelli-rosenberg1} \cite{nelli-rosenberg2}, \cite{nelli-rosenberg3} and the references therein.

Recently, the theory  minimal and constant mean curvature hypersurfaces started to be  developed in more general spaces, as in  the work of S. Montiel \cite{montiel-indiana} and Al\'{\i}as-Dajczer \cite{alias-dajczer-1}, \cite{alias-dajczer-2} where they  studied constant mean curvature hypersurfaces in warped product manifolds $M^{n+1}=\mathbb{R}\times_{\varrho}\mathbb{P}^{n}$, where $\mathbb{P}^{n}$ is a complete Riemannian manifold and $\varrho \colon \mathbb{R}\to \mathbb{R}_{+}$ is a smooth warping function. Those studies were further extended  by Al\'{i}as, Dajczer and Rigoli \cite{alias-dajczer-rigoli} and Al\'{i}as, Impera and Rigoli \cite{alias-impera-rigoli} to include constant higher order mean curvature  hypersurfaces in warped product manifolds and in general setting by Albanese, Al\'{\i}as and Rigoli  \cite{albanese-alias-rigoli}.

In this paper we  give a small contribution to the theory proving appropriate extensions the results of \cite{alias-impera-rigoli}. We start in Section \ref{sec1}
  presenting general conditions for the validity of the weak  maximum principle for a fairly general class of semi-elliptic trace operators \ref{operator}, see Theorem \ref{Thm1}. These operators and versions of Omori-Yau maximum principle were considered by Alias, Impera and Rigoli \cite{alias-impera-rigoli} and by Hong and Sung \cite{hong-sung} under slightly more restrictive conditions.
  Then, we derive few geometric conditions on a manifold  that  guarantee that Theorem \ref{Thm1} applies, see Corollary \ref{corollary1} and Theorem \ref{corollary2}.  In section \ref{sec2} we consider   the $L_r$ operators and  prove several  higher order mean curvature estimates for hypersurfaces immersed into warped product spaces $M^{n+1}=\mathbb{R}\times_{\varrho}\mathbb{P}^{n}$, see Theorems \ref{Thm2}, \ref{Thm3}. In section \ref{sec3} we extend the  Slice Theorem of Al\'{i}as-Impera-Rigoli \cite[Thms. 16 \& 21]{alias-impera-rigoli}, see Theorems \ref{thm5}, \ref{thm6}.

 \section{Maximum Principle for Trace Operators}\label{sec1}Following the terminology introduced in \cite{prs-memoirs} we say that the Omori-Yau maximum principle for the Laplacian holds on $M$ if for any given $u\in C^{2}(M)$ with $u^{\ast}=\sup_{M}u<+\infty$, there exists a sequence of points $x_{k}\in M$, depending on $M$ and  $u$,  such that
\begin{eqnarray}\label{eq-omori-2}
\lim_{k\to +\infty}u(x_{k}) =u^{\ast},& \vert \grad u\vert  (x_{k}) < \displaystyle \frac{1}{k},&  \triangle u(x_{k}) <  \displaystyle \frac{1}{k}.
\end{eqnarray} Likewise, the Omori-Yau maximum principle \textit{for the Hessian} is said to hold on $M$ if for any given $u\in C^{2}(M)$ with $u^{\ast}=\sup_{M}u<+\infty$, there exists a sequence of points $x_{k}\in M$, depending on $M$ and on $u$,  such that \begin{eqnarray}\label{eq-omori}
\lim_{k\to +\infty}u(x_{k}) =u^{\ast},& \vert \grad u\vert  (x_{k}) < \displaystyle \frac{1}{k},&  \Hess u(x_{k})(X,X) <  \displaystyle \frac{1}{k}\cdot \vert  X \vert^{2},
\end{eqnarray}for every $X\in T_{x_k}M$.  Accordingly,  the classical results of Omori \cite{omori} and Yau \cite{yau-comm-pure-75} can be stated saying that the Omori-Yau maximum principle for the Laplacian  holds on  Riemannian manifold with Ricci curvature bounded from below. The importance of the Omori-Yau maximum principle lies on its  wide range of  applications in geometry and analysis. Applications  that goes  from the generalized Schwarz lemma \cite{yau2} to the study of the group of conformal diffeomorphism of a manifold \cite{PRS-PAMS}, from  curvature estimates on submanifolds \cite{albanese-alias-rigoli}, \cite{alias-bessa-dajczer-MathAnn}, \cite{alias-bessa-montenegro} to Calabi conjectures on minimal hypersurfaces \cite{jorge-xavier-bms}, \cite{omori}. The essence of the Omori-Yau maximum principle was captured by Pigola, Rigoli and Setti in Theorem 1.9 of   \cite{prs-memoirs}  whose corollary is the following result: The Omori-Yau maximum principle holds on every Riemannian manifold $M$ with Ricci curvature satisfying $Ric_{M}(\grad \rho, \grad \rho) \geq - C^{2}G(\rho) $ where $\rho $ is the distance function on $M$ to a point, $C$ is positive constant and  $ G:[0,+\infty)\to [0, \infty) $ is a smooth function  satisfying $ G(0) > 0, \,  G'(t) \geq 0$, $ \displaystyle\int_{0}^{+\infty}\frac{ds}{\sqrt{G(s)}} = +\infty$ and $\displaystyle\limsup_{t \rightarrow +\infty}\frac{tG(\sqrt{t})}{G(t)} < +\infty$.

In most applications of the Omori-Yau maximum principle, the  condition $\vert \grad u\vert (x_k)< 1/k$ is redundant, which led to the following definition.
\begin{definition}[Pigola-Rigoli-Setti] The weak  maximum principle holds on a Riemannian manifold
$M$ if for every  $u\in C^{2}(M)$ with  $u^{\ast}=\sup_{M}u<+\infty$, there exists a sequence of points $x_{k}\in M$,   such that
\begin{eqnarray}\label{eq-omori-3}
\lim_{k\to +\infty}u(x_{k}) =u^{\ast},&   \triangle u(x_{k}) <  \displaystyle \frac{1}{k}.
\end{eqnarray}\end{definition}
 This apparently simple minded  definition   proved to be surprisingly deep. For instance,  it  has been proven in  \cite{PRS-PAMS},  that the weak  maximum principle for the  Laplacian is equivalent to the stochastic completeness of the diffusion process associated
to $\triangle$.

The Omori-Yau/weak maximum principle  can be considered  for  differential elliptic operators other than the Laplacian,  like  the $\phi$-Laplacian \cite{prs-revista-2006}, the weighted Laplacian $\triangle_{f}=e^{f} \diver (e^{-f}\grad)$ \cite{bps-revista}, \cite{barnabe-leandro} and semi-elliptic  trace operators $L=Tr (P \circ \hess )$  considered in  \cite{alias-impera-rigoli}, \cite{pessoa-lima1} and in  \cite{bjlm-abc},   where $P\colon TM\to TM$ is a positive semi-definite symmetric tensor on $TM$ and for each $u \in C^{2}(M)$,  $\hess u\colon TM \to TM$ is a symmetric  operator defined by $\hess u (X)=\nabla_{X}\grad u$ for every $X\in TM$. Here $\nabla$ be the Levi-Civita connection of $M$.
In this paper we are going to consider the trace operator \begin{equation}\label{operator}Lu=Tr (P \circ \hess u ) + \langle V,\grad u \rangle \end{equation} with $\limsup_{x\to \infty}\vert V\vert (x)< +\infty$, and prove an Omori-Yau maximum principle in the same spirit of Theorem 1.9 of   \cite{prs-memoirs}. Then we prove some geometric applications that extends those of \cite{alias-impera-rigoli}. In \cite{hong-sung},  K. Hong and C. Sung considered  the same trace operator and proved an Omori-Yau maximum principle but their proof required  the  stronger condition  $\sup_{M}Tr (P)+ \sup \vert V\vert <\infty.$
 It should  be pointed out that Albanese, Alias and Rigoli  in \cite{albanese-alias-rigoli} recently proved  an all general Omori-Yau maximum principle for trace operators of the form $ Lu=Tr (P \circ \hess u ) + \diver P (\grad u)+ \langle V,\grad u \rangle$. However,  the main purpose of this paper is to extend the  geometric applications involving the operator \eqref{operator} proved in \cite{alias-impera-rigoli} and for the sake of completeness we keep the proof of our version of the Omori-Yau maximum principle, besides it is very simple.

 Our first result is the following extension of \cite[Thm. 1]{alias-impera-rigoli}.

\begin{theorem}\label{Thm1}
Let $ (M,\langle,\rangle) $ be a complete Riemannian manifold. Consider a semi-elliptic operator $ \Ll =  \tr(\Pp\circ\hess (\cdot)) + \langle V, \grad (\cdot) \rangle $, where $ \Pp : TM \rightarrow TM $ is a positive semi-definite symmetric tensor and $ V $ satisfies $ \limsup_{x \to \infty} \vert V\vert (x) < +\infty $. Suppose that exists a non-negative function $\gamma\in C^2(M) $ satisfying:
\begin{enumerate}
\item[i)] $ \gamma(x) \rightarrow + \infty $ as $ x \rightarrow \infty $,
\item[ii)] $ \exists A > 0 $ such that $\displaystyle{\vert \nabla \gamma\vert \leq A \sqrt{G(\gamma)}\cdot\left(\int_{0}^{\gamma}\frac{1}{\sqrt{G(s)}}ds+1\right)}$ off a compact set,
\item[iii)] $ \exists B > 0 $ such that  $ \displaystyle{\tr(\Pp\circ\hess \gamma) \leq B \sqrt{G(\gamma)}\cdot \left(\int_{0}^{\gamma}\frac{ds}{\sqrt{G(s)}} + 1\right)} $ off a compact set.
\item[iv)]Where $ G : [0,\infty) \rightarrow [0,\infty) $ is such that:
$\begin{array}{llll} G(0) > 0  &  G'(t) \geq 0  &{\rm and}& \displaystyle{G(t)^{-\frac{1}{2}} \not\in L^{1}(+\infty)} .\end{array}$
\end{enumerate}
Then given any function $ u \in C^2(M) $ that satisfies
\begin{eqnarray}\label{condictionu}
\displaystyle{\lim_{x \rightarrow \infty}\frac{u(x)}{\varphi(\gamma(x))}} = 0,
\end{eqnarray}
where
\begin{eqnarray}\label{defvarphi}
\displaystyle{\varphi(t) = \ln\left(\int_{0}^{t}\frac{ds}{\sqrt{G(s)}} + 1\right)},
\end{eqnarray}
there exists a sequence $ \{x_k\}_{k \in \mathbb{N}} \subset M $ satisfying:
\begin{eqnarray*}
(a) \ \ \vert \nabla u\vert(x_k) < \frac{1}{j} \qquad \mbox{and} \qquad (b) \ \ \Ll u(x_k) < \frac{1}{j} .
\end{eqnarray*}

If instead of \eqref{condictionu} we suppose that $ u $ is bounded above  we have  that
\begin{eqnarray*}
(c) \lim_{k \rightarrow +\infty}u(x_k) = \sup u .
\end{eqnarray*}
\end{theorem}

\subsection{Proof of Theorem \ref{Thm1}}

The proof of Theorem \ref{Thm1}  we will follow the same steps of the proof of  the maximum principle of Omori-Yau in the case of the Laplacian and the f-Laplacian presented in \cite{barnabe-leandro}. Thus, we only prove the item (b).
Consider the following family of functions
\begin{eqnarray*}
f_{k}(x) = u(x) - \varepsilon_{k} \varphi(\gamma(x)),
\end{eqnarray*}
where $ \varphi $ is defined in \eqref{defvarphi} and $ \varepsilon_{k} \rightarrow 0^{+} $ when $ k \rightarrow + \infty $. Observe  that the condition  \eqref{condictionu} implies that $ f_{k} $ reaches a local maximum, say at $ x_{k} \in M $. Suppose that the sequence $ \{x_k\}_{k \in \mathbb{N}} $ diverges, (leaves any compact subset of $ M$) otherwise we have nothing to prove.

Using the fact that $ x_{k} $ be the point of maximum to $ f_{k} $ we infer that
\begin{eqnarray}\label{expression0}
0& =& \grad u(x_k) - \varepsilon_{k}\varphi{'}(\gamma(x_k))\grad \gamma (x_k)
\end{eqnarray}
and for all $v\in T_{k}M$ we have
\begin{eqnarray}\label{expression1}
0& \geq &\Hess u(x_k)(v,v) - \varepsilon_{k}\left[\varphi{'}(\gamma(x_k))\Hess \gamma(x_k)(v,v) - \varphi{''}(\gamma(x_k))\langle \grad \gamma(x_k),v\rangle^{2} \right].
\end{eqnarray}
Calculating $ \varphi'$ and $ \varphi'' $ we obtain
\begin{eqnarray}\label{varphilinha}
\varphi^{'}(t)= \left\{\left(\int_{0}^{t}\frac{ds}{\sqrt{G(s)}} + 1\right)\sqrt{G(t)}\right\}^{-1}
\end{eqnarray}
and
\begin{eqnarray}\label{varphiduaslinhas}
\varphi^{''}(t) &=& - \left\{\sqrt{G(t)} \left(\int_{0}^{t}\frac{1}{\sqrt{G(s)}}ds +1  \right) \right\}^{-2}\left\{\frac{G'(t)}{2\sqrt{G(t)}} \left( \int_{0}^{t}\frac{1}{\sqrt{G(s)}}ds+1\right)+1  \right\}\leq 0
\end{eqnarray}
Taking \eqref{varphilinha} into \eqref{expression0}, we have by the hypothesis ii)
\begin{eqnarray}\label{graduxn}
\vert \grad u\vert(x_k) \leq \varepsilon_{k}\vert \varphi{'}(\gamma(x_k))\vert \vert \grad \gamma\vert(x_k) \leq \varepsilon_{k}.
\end{eqnarray}
Taking  \eqref{varphiduaslinhas} into \eqref{expression1}, we get
\begin{eqnarray}\label{expression2}
\Hess u(x_{k})(v,v) \leq \varepsilon_{k}\varphi{'}(\gamma(x_k))\Hess \gamma(x_k)(v,v).
\end{eqnarray}

Choose a basis of eigenvectors $ \{v_{1},...,v_{n}\} \subset T_{x_k} M $  of  $ P(x_k) $, corresponding to the  eigenvalues $ \lambda_{j}(x_k) = \langle P(x_k)v_{j},v_{j} \rangle \geq 0 $, with $ 1 \leq j \leq n={\rm dim}(M) $. Therefore, by the inequality \eqref{expression2} it follows that
\begin{eqnarray*}
\langle P(x_k)\hess u(x_k)v_{j},v_{j}\rangle &=& \lambda_{j}(x_k)\Hess u(x_k)(v_{j},v_{j}) \\
& \leq & \varepsilon_{k}\varphi{'}(\gamma(x_k))\langle P(x_k)\hess \gamma(x_k)v_{j},v_{j} \rangle
\end{eqnarray*}
Applying the trace on both sides of the inequality above and using \eqref{varphilinha} with the hypothesis $iii)$, we have
\begin{eqnarray}
\tr(\Pp\circ\hess u(x_k)) \leq \varepsilon_{k}\varphi{'}(\gamma(x_k))\tr(\Pp\circ\hess \gamma(x_k)).
\end{eqnarray}
Therefore, by the item (a) and that $ \limsup_{x \to \infty}\vert V\vert(x) < +\infty $ we have
\begin{eqnarray*}
\Ll u(x_k) &=& \tr(\Pp\circ\hess u(x_k)) + \langle V(x_k),\grad u(x_k)\rangle \\
& \leq & \varepsilon_{k}\varphi{'}(\gamma(x_k))\tr(\Pp\circ\hess \gamma(x_k)) + D\vert \grad u(x_k)\vert \\
& \leq & (D+1)\varepsilon_{k}.
\end{eqnarray*}
This finish the proof of Theorem \ref{Thm1}.
\begin{remark} \label{remark1} The   Theorem \ref{Thm1} can be proved substituting the conditions ii) and iii)  by the apparently more general conditions: \begin{itemize}\item[i.] $\displaystyle{\vert \nabla \gamma\vert \leq A\prod_{j=1}^{\ell}\left[\ln^{(j)}\left(\int_{0}^{t}\frac{ds}{\sqrt{G(s)}} + 1\right)+1\right]\left(\int_{0}^{t}\frac{ds}{\sqrt{G(s)}} + 1\right)\sqrt{G(t)}} $ and \item[ii.] $ \displaystyle{\tr(\Pp\circ\hess \gamma) \leq B\prod_{j=1}^{\ell}\left[\ln^{(j)}\left(\int_{0}^{t}\frac{ds}{\sqrt{G(s)}} + 1\right)+1\right]\left(\int_{0}^{t}\frac{ds}{\sqrt{G(s)}} + 1\right)\sqrt{G(t)}} $
\end{itemize} respectively. Just consider in the proof $\varphi (t) = \log^{(\ell+1)} \left(\int_{0}^{t}\frac{ds}{\sqrt{G(s)}} + 1\right)$.
\end{remark}

\begin{corollary}\label{corollary1}
Let $ (M, \langle, \rangle) $ be a complete, non-compact, Riemannian manifold with radial sectional curvature satisfying
\begin{eqnarray}\label{sectionalcurvature}
K_{M} \geq - B^{2}\prod_{j=1}^{\ell}\left[\ln^{(j)}\left(\int_{0}^{r}\frac{ds}{\sqrt{G(s)}} + 1\right)+1\right]^{2}G(r),\,\,\,\,{\rm for} \,\,\,\, r(x)\gg 1,
\end{eqnarray}
 where $G\in C^{\infty}([0,+\infty)) $ is even at the origin and satisfies iv), $r(x)={\rm dist}_{M}(x_{0}, x)$ and  $B\in \mathbb{R}$. Then, the Omori-Yau maximum principle for any semi-elliptic operator $ \Ll =  \tr(\Pp\circ\hess (\cdot)) + \langle V, \grad (\cdot) \rangle $ with  $ \tr\Pp \leq \left(\int_{0}^{r}\frac{ds}{\sqrt{G(s)}} + 1\right) $ and $\limsup_{x \to \infty}\vert V \vert < + \infty$  holds on $M$.
\end{corollary}
\begin{proof}
Following the same steps of the example \cite[Example 1.13]{prs-memoirs} one has that bound \eqref{sectionalcurvature} implies
$$ \Hess r \leq D \prod_{j=1}^{\ell}\left[\ln^{(j)}\left(\int_{0}^{r}\frac{ds}{\sqrt{G(s)}}+1\right)+1\right]\sqrt{G(r)}$$ and this implies $$ \tr(P\circ \hess) \leq D \prod_{j=1}^{\ell}\left[\ln^{(j)}\left(\int_{0}^{r}\frac{ds}{\sqrt{G(s)}}+1\right)+1\right]\left(\int_{0}^{r}\frac{ds}{\sqrt{G(s)}}+1\right)\sqrt{G(r)}.$$ Hence Theorem \ref{Thm1} applies.

\end{proof}

\section{Immersions into Warped Products}By $ L^{\ell}\times_{\rho}P^n = N^{n+\ell}  $ we denote the product manifold $L^{\ell}\times P^n$ endowed with the  warped product metric $dL^{2}+ \rho^{2}(x) dN^{2}$, where $ L^{\ell} $ and $ P^n $ are Riemannian manifolds and $ \rho : L \to \mathbb{R}_{+} $ is a positive smooth function.
 We will need the following  definition   introduced in \cite{alias-bessa-montenegro-piccione}. \begin{definition} \label{def2}Let $M$ be a Riemannian manifold. We say that $ (G,\tilde{\gamma}) $ is an Omori-Yau pair for the Hessian in $M$ if $ \tilde{\gamma}\in C^{2}(M) $ is proper and satisfies
\begin{eqnarray}\label{eq1}
\vert \grad \tilde{\gamma} \vert &\leq & \sqrt{G(\tilde{\gamma})}\left(\int_{0}^{\tilde{\gamma}}\frac{ds}{\sqrt{G(s)}}+1\right)\nonumber  \\
&&\\
\Hess \tilde{\gamma} &\leq & \sqrt{G(\tilde{\gamma})}\left(\int_{0}^{\tilde{\gamma}}\frac{ds}{\sqrt{G(s)}}+1\right)\nonumber
\end{eqnarray}
\end{definition} \begin{remark} If a Riemannian manifold $M$ has an Omori-Yau pair for the Hessian  $ (G,\tilde{\gamma}) $ then the Omori-Yau maximum principle for the Hessian holds on $M$, see \cite{alias-bessa-montenegro-piccione}, \cite{barnabe-leandro}.\end{remark}

The following result gives conditions for an isometric immersion $ f \colon M^{m} \rightarrow L^{\ell}\times_{\rho}P^n = N^{n+\ell} $ into a warped product, where $ P^n $ carries an Omori-Yau pair $(G,\tilde{\gamma}) $ for the Hessian, to  carry an Omori-Yau pair.

\begin{theorem}\label{corollary2}
Let $ f \colon M^{m} \rightarrow L^{\ell}\times_{\rho}P^n = N^{n+\ell} $ be an isometric immersion where $ P^n $ carries an Omori-Yau pair $(G,\tilde{\gamma}) $ for the Hessian, $ \rho \in C^{\infty}(L) $ is a positive function, such that $ \inf \rho > 0 $ and $ \mathcal{H} = \grad \log \rho $ satisfies
\begin{eqnarray}\label{hipeta}
\vert \mathcal{H}\vert(\pi_{L}(f)) \leq \ln\left(\int_{0}^{\pi_{L}(f)}\frac{ds}{\sqrt{G(s)}}+1\right).
\end{eqnarray}
If $ f $ is proper on the first entry and
\begin{eqnarray}\label{hipcurvmedia}
\vert \alpha \vert \leq \ln\left(\int_{0}^{\tilde{\gamma}\circ\pi_{P}(f)}\frac{ds}{\sqrt{G(s)}}+1\right),
\end{eqnarray}
where $ \pi_{L}, \pi_{P} $ are the projections on $ L^{\ell} $ and $ P^n $ respectively, and $\alpha$ is the second fundamental form of the immersion $f$, then $ M^m $ has an Omori-Yau pair for any semi-elliptic operator \begin{eqnarray*} \Ll =  \tr(\Pp\circ\hess (\cdot)) + \langle V, \grad (\cdot) \rangle & with &\displaystyle{\tr\Pp \leq \prod_{j=2}^{k}\left[\ln^{(j)}\left(\int_{0}^{\tilde{\gamma}\circ\pi_{P}(f)}\frac{ds}{\sqrt{G(s)}} + 1\right)+1\right]} . \end{eqnarray*}
\end{theorem}
\begin{proof}
By abuse of language,  we will  denote  by $ \langle , \rangle_{N}$,   $ \langle , \rangle_{L}$,   $ \langle , \rangle_{P}$  and $ \langle , \rangle_{M}$ the Riemannian metrics on $ N^{n+\ell} $, $ L^{\ell} $,  $ M^{m} $ and  on  $ P^n $ respectively. Let
$ \Vert\, \cdot \, \Vert_{N,L,M.P} $  be their respective norms.  We will denote by $ X,Y $,  vector fields in   $ TL $ and by  $ W,Z $  vector fields in $ TP $.  Let $\nabla^{N}$, $\nabla^{P}$ and $\nabla^{L}$ denote the Riemannian connections on $N$, $P$ and $L$ respectively. We need few lemmas to prove Corollary \ref{corollary2}.
\begin{lemma}The proof of the following relations are straight forward.
\begin{equation}\begin{array}{lllllllll}\nabla^{N}_{X}Y &=& \nabla^{L}_{Y}X, & &\nabla^{N}_{Z}X& = &\nabla^{N}_{X}Z &=& X(\eta)Z,\end{array}
\begin{array}{lll}\nabla^{N}_{Z}W& =& \nabla^{P}_{Z}W - \langle Z,W\rangle\displaystyle \grad^{L}\eta ,\end{array}\end{equation}where  $ \eta = \log \rho $.
\end{lemma}
Recall that $P$ carries a Omori-Yau pair $(G, \tilde{\gamma})$. Letting $\pi_{P}\colon N^{\ell +n} \to P^{n}$ be the projection on the second factor we define  $ \beta \colon N \to \mathbb{R} $ by  $ \beta = \tilde{\gamma}\circ \pi_{P} $ e $ \gamma \colon M \to \mathbb{R} $ by $ \gamma = \beta \circ f $. It is clear that
\begin{equation}\begin{array}{l}
\langle \grad \beta , X \rangle_{N} = 0  \end{array}{\rm and} \begin{array}{l} \langle \grad \tilde{\gamma}, Z\rangle_{P} = \langle \grad \beta , Z \rangle_{N}
= \rho^{2}\langle\grad \beta , Z\rangle_{P}.\end{array}
\end{equation}Thus  $\grad \beta = \displaystyle\frac{1}{\rho^2}\grad \tilde{\gamma}$, $\Vert \grad \tilde{\gamma}\Vert_{N} = \rho \Vert \grad \tilde{\gamma}\Vert_{P}$
and $$\vert \grad \gamma \vert_{M} \leq \Vert \grad \beta \Vert_{N}=\frac{1}{\rho^{2}}\Vert \grad \tilde{\gamma}\Vert_{N}=\frac{1}{\rho}\Vert \grad \tilde{\gamma}\Vert_{P}.$$ Moreover, for all $e\in TM$ we have that, (identifying $f_{\ast}e=e$), \begin{equation}\label{eq19A}\Hess_{M} \gamma (e,e) = \Hess_{N} \beta (e,e) + \langle \grad \beta ,\alpha(e,e)\rangle_{N}. \end{equation} Here $\alpha$ is the second fundamental form of the immersion.
Let us write $e=X+Z$ where $X\in TL$ and $Z\in TP$  and  we have that $\Hess_{N} (e,e)=\Hess_{N} (X,X)+ 2\Hess_{N} (X,Z) + \Hess_{N} (Z,Z)$ and \begin{equation}\Hess_{N} \beta (X,X) =\langle \nabla^{N}_{X}\grad \beta ,X\rangle_{N}= \langle X(\eta)\grad \beta , X\rangle_{N}=0.\label{eq21}\end{equation}

\begin{equation}\Hess_{N} \beta (X,Z) = \langle \nabla^{N}_{X} \grad \beta ,Z\rangle_{N}
=X(\eta)\langle \grad \beta ,Z \rangle_{N}
= \langle \grad \eta ,X\rangle_{L} \langle \grad \tilde{\gamma},Z\rangle_{P} .\label{eq22}\end{equation}

\begin{eqnarray}\Hess_{N} \beta (Z,Z) &=& \langle \nabla^{N}_{Z}\grad \beta ,Z\rangle_{N} \nonumber \\
&=& \langle \nabla^{P}_{Z}\grad \tilde{\gamma}, Z\rangle_{N}  - \langle \grad \beta, Z\rangle_{N} \cancelto{0}{\langle  \grad \eta ,Z\rangle_{N}}\label{eq23} \\
& = & \frac{1}{\rho^2}\langle \nabla^{P}_{Z}\grad \tilde{\gamma} ,Z \rangle_{N}
= \Hess_{P} \tilde{\gamma} (Z,Z).\nonumber
\end{eqnarray}

Hence from (\ref{eq21}, \ref{eq22}, \ref{eq23}) we have that
\begin{eqnarray}
\Hess_{N} \beta (e,e) = 2\langle \grad \eta ,X\rangle_{L}\langle \grad \tilde{\gamma},Z\rangle_{P} + \Hess_{P} \tilde{\gamma}(Z,Z).
\end{eqnarray}
Thus, from \eqref{eq19A},
\begin{eqnarray}
\Hess_{M} \gamma (e,e) &=& 2\langle \grad \eta ,X\rangle_{L}\langle \grad \tilde{\gamma},Z\rangle_{P} + \Hess_{P} \tilde{\gamma}(Z,Z) + \langle \grad \beta ,\alpha(e,e)\rangle_{N} \nonumber \\
&&\nonumber \\
&=&  2\langle \grad \eta ,X\rangle_{L}\langle \grad \tilde{\gamma},Z\rangle_{P} + \Hess_{P} \tilde{\gamma}(Z,Z) + \frac{1}{\rho^2}\langle \grad \tilde{\gamma},\alpha(e,e)\rangle_{N}\nonumber  \\
&& \\
&\leq & 2\Vert \mathcal{H}\Vert_{L} \cdot \Vert \grad \tilde{\gamma}\Vert_{P} \cdot \Vert X\Vert_{L}\cdot\Vert Z\Vert_{P} + \Hess_{P} \tilde{\gamma}(Z,Z) + \frac{1}{\rho^2}\langle \grad \tilde{\gamma},\alpha(e,e)\rangle_{N} \nonumber \\
&&\nonumber \\
&\leq & \ln\left(\int_{0}^{\gamma}\frac{ds}{\sqrt{G(s)}}+1\right)\sqrt{G(\gamma)}\left(\int_{0}^{\gamma}\frac{ds}{\sqrt{G(s)}}
+1\right)\left(3+\frac{1}{\rho}\right)\vert e\vert^2 \nonumber
\end{eqnarray}off a compact set,  since $\ln\left(\int_{0}^{\gamma}\frac{ds}{\sqrt{G(s)}}+1\right)>1$ there.

\vspace{2mm}

Let  $ x \in M $ and choose a basis  $ \{e_{1},...,e_{n}\}$ for $T_{x}M$ formed by eigenvectors of $P(x)$ with eigenvalues $ \lambda_{j}(x) = \langle P(x)e_{j},e_{j} \rangle \geq 0 $. Since $$\Ll \gamma (x)=\sum_{i=1}^{n} \langle P(x)(\hess \gamma(x)(e_i)), e_i\rangle = \sum_{i=1}^{n} \langle  \hess \gamma(x)(e_i), P(x)(e_i)\rangle= \sum_{i=1}^{n} \lambda_i (x) \Hess \gamma(e_i, e_i) $$ we have then
\begin{eqnarray*}
\Ll \gamma &\leq & \tr P\cdot \ln\left(\int_{0}^{\gamma}\frac{ds}{\sqrt{G(s)}}+1\right)\sqrt{G(\gamma)}\left(\int_{0}^{\gamma}\frac{ds}{\sqrt{G(s)}}
+1\right)\left(3+\frac{1}{\rho}\right) + \frac{1}{\rho}\vert V \vert\sqrt{G(\gamma)}\left(\int_{0}^{\gamma}\frac{ds}{\sqrt{G(s)}}+1\right)  \nonumber \\
&\leq & \sqrt{G(\gamma)}\left(\int_{0}^{\gamma}\frac{ds}{\sqrt{G(s)}}+1\right)\left(3+\frac{D}{\rho}\right)\prod_{j=1}^{k}\ln^{(j)}\left(\int_{0}^{\gamma}
\frac{ds}{\sqrt{G(s)}} + 1\right)
\end{eqnarray*}By Theorem \ref{Thm1} and Remark \ref{remark1} the Omori-Yau maximum principle holds on $M$.
\end{proof}

 \section{Curvature estimates}\label{sec2}

\subsection{The operators $L_r$}Let $ f \colon M\hookrightarrow N$ be an isometric immersion of a  connected $n$-dimensional Riemannian manifold $M$ into the $(n+1)$-dimensional Riemannian manifold $N$. Let $\alpha(X)=-\overline{\nabla}_{X}\eta$, $X\in TM$ be the second fundamental form of the immersion with respect to a locally defined normal vector field $\eta$, where $\overline{\nabla}$ is
the Levi-Civita connection of $N$. Its eigenvalues $\kappa_{1},\ldots, \kappa_{n}$ are the principal curvatures of the hypersurface $M$. The elementary symmetric functions of the principal curvatures are defined by

\begin{equation}
S_{0}=1,\;\;\;\;\;S_{r}= \sum_{i_{1}<\cdots <i_{r}} k_{i_{1}} \cdots
k_{i_{r}},\;\;\;\;\;1\leq r\leq n. \label{Sr-eq1}
\end{equation}

The elementary symmetric functions $S_r$ define the $r$-mean curvature $H_r$ of the immersion by  \begin{equation}H_r=\left(\begin{array}{l}n\\r \end{array}\right)^{-1}S_{r},
\end{equation} so that $H_{1}$ is the mean curvature and $H_{n}$ is the Gauss-Kronecker curvature.
The Newton tensors $P_r\colon TM \to TM$, for $r =
0,1,\ldots, n$, are defined setting
$P_{0}=I\,\,\,{\rm and }\,\, P_{r} = S_{r}Id-\alpha P_{r-1}$ so that
$P_{r}(x): T_{x}M\rightarrow T_{x}M$ is a self-adjoint linear
operator  with the same eigenvectors as $\alpha$. From here  we will be following \cite[p.3]{alias-impera-rigoli} closely. When $r$ is even the sign of $S_r$ does not depend on $\eta$ which implies that the tensor $P_r$ is globally defined on $TM$. If  $r$ is odd we will assume that $M$ is two sided, i.e. there exists a globally defined unit normal vector field $\eta$ in $f(M)$. When a hypersurface is two sided, a choice of $\eta $ makes $P_r$ globally defined. To give an uniform treatment in what follows we shall assume from now on that $M$ is two-sided.

For each $u\in C^{2}(M)$,   define a symmetric operator $\hess u\colon TM \to TM$  by \begin{equation}\hess u (X)=\nabla_{X}\grad u\end{equation} for every $X\in TM$, where  $\nabla$ be the Levi-Civita connection of $M$ and the symmetric bilinear form  $\Hess u \colon TM \times TM \to C^{\infty}(M)$ by \begin{equation}\Hess u (X,Y)=\langle \hess u (X), Y\rangle.\end{equation} Associated to  each Newton operator $P_r\colon TM \to TM $ there is  a second order self-adjoint differential operator $L_r\colon C^{\infty}(M)\to C^{\infty}(M)$ defined by
\begin{equation} L_r (u) = Tr(P_r \hess u)=\diver (P_r \grad u)-\langle Tr (\nabla P_r), \grad u\rangle    \end{equation} However, these operators may be not elliptic. Regarding the ellipticity of the $L_r$'s, one sees  that the operator $L_r$ is elliptic if and only if $P_r$ is positive definite. There
 are geometric conditions implying the positiveness of the $P_r$ and thus the ellipticity of the
$L_{r}$, e.g. $H_{2}>0$ implies that $H_{1}>0$ by the well known inequality $H_{1}^{2}\geq H_2$. And that implies that all the eigenvalues of $P_1$ are positive and  the ellipticity of $L_{1}$. For  the ellipticity of $L_r$, $r\geq 2$, it is enough to assume that there exists  an elliptic point $p\in M$, i.e. a point where the second fundamental form $\alpha$ is positive definite (with respect to an orientation) and $H_{r+1} >0$. See details in \cite{barbosa-colares}, \cite{cheng-yau-math-ann-77}, \cite{hartaman-TAMS-78}, \cite{reilly-jdg-73}.
 In this section we are going to apply Theorem \ref{Thm1} to the operators $L_r$ in order to derive curvature estimates.

 Again, we denote by $ N^{n+1} = I\times_{\rho}P^n $ the product manifold $I\times P^n$ endowed with the warped product metric $dt^{2}+\rho^{2}(t)d^{2}P$, where $ I \subset \mathbb{R} $ is a open interval, $ P^n $ is a complete Riemannian manifold and $ \rho \colon I \to \mathbb{R}_{+} $ is a smooth function. Given an isometrically immersed hypersurface  $ f \colon  M^{n} \to N^{n+1} $, define $ h \colon M^{n} \to I $ the $ C^{\infty}(M^{n}) $ height function by setting $ h = \pi_{I}\circ f $, where $\pi_{I}\colon N^{n+1}=I\times_{\rho}P^n\to I$ is the projection on the first factor. We will need the following result similar to \cite[Prop. 6]{alias-impera-rigoli}.


\begin{lemma}\label{lemma2}
Let $ f \colon  M^n \to I\times_{\rho}P^n = N^{n+1} $ be an isometric immersion into a warped product space. Let $ h $ be the height function and define
\begin{eqnarray}\label{fsigma}
\sigma(t) &=& \int_{t_0}^{t}\rho(s)ds.
\end{eqnarray}
Then
\begin{eqnarray}
\hat{\Ll}_{k}\sigma(h) = c_{k}\rho(h)\left(\mathcal{H}(h) + \Theta\frac{H_{k+1}}{H_k}\right),
\end{eqnarray}
where $ c_k = (n-k)\left(\begin{array}{c}
n \\
k
\end{array}\right) $, $ \Theta = \langle \eta,T \rangle $ is the angle function,  $ \mathcal{H}(h) =\displaystyle \frac{\rho'(h)}{\rho(h)} $ and $ \hat{\Ll}_k = \tr(\hat{\Pp}_{k}\circ \hess) $ with $ \hat{\Pp}_k = \dfrac{\Pp_k}{H_k} $.
\end{lemma}
\begin{proof}
We observe that $ \grad \sigma(h) = \rho(h)\grad h $ and consequently
\begin{eqnarray}\label{hesssigma}
\hess \sigma(h)(X) &=& \nabla_{X}\grad \sigma(h) \nonumber \\
&& \nonumber \\
&=& \rho(h)\nabla_{X}\grad h + \langle \grad(p\circ h),X\rangle \grad h \\
&& \nonumber\\
&=& \rho(h)\hess h(X) + \rho'(h)\langle \grad h,X\rangle \grad h , \nonumber
\end{eqnarray}
for all $ X \in TM^n $.
Therefore,
\begin{eqnarray}\label{eqLl1}
\hat{\Ll}_{k}\sigma(h) &=& \tr(\hat{P}_{k}\circ\hess \sigma(h)) \nonumber \\
&=& \sum_{i}^{n}\langle \frac{1}{H_k}P_{k}\circ\hess \sigma(h)(e_i),e_i\rangle \nonumber \\
&=& \frac{1}{H_k}\sum_{i}^{n}\langle \rho(h)P_{k}\circ\hess h(e_i) + \rho'(h)\langle \grad h,e_i \rangle P_{k}\grad h,e_i\rangle \\
&=& \frac{1}{H_k}\left(\rho(h)\sum_{i}^{n}\langle P_{k}\circ\hess h(e_i),e_i\rangle + \rho'(h)\langle \grad h,P_{k}\grad h\rangle\right). \nonumber
\end{eqnarray}

For the other hand, the gradient of $ \pi_{I} \in C^{\infty}(M) $ is $ \grad^{N}\pi_{I} = T $, where $ T $ stands for the lifting of $\partial/ \partial t \in TI$
 to the product $I\times_{\rho}P^{n}$. Then,
\begin{eqnarray}
\grad h = (\grad^{N}\pi_{I})^{\perp} = T - \Theta \eta.
\end{eqnarray}
Since the Levi-Civita connection of a warped product satisfies
\begin{eqnarray*}
\bar{\nabla}_{X}T = \mathcal{H}(X - \langle X,T\rangle T) , \ \ \ \forall X \in TN^{n+1}
\end{eqnarray*}
we have
\begin{eqnarray*}
\bar{\nabla}_{X}\grad h = \mathcal{H}(h)(X - \langle X,T\rangle T) - X(\Theta)\eta + \Theta AX , \ \ \ \forall X \in TM^n
\end{eqnarray*}
and thus
\begin{eqnarray}\label{eqhessh}
\hess h(X) = \mathcal{H}(h)(X - \langle X,\grad h\rangle \grad h) + \Theta AX
\end{eqnarray}

Taking   \eqref{eqhessh} in account in \eqref{eqLl1} we get
\begin{eqnarray*}
\hat{\Ll}_{k}\sigma(h) &=& \frac{1}{H_k}\left[\rho(h)\left(\mathcal{H}(h)(\tr P_{k} - \langle P_{k}\grad h,\grad h\rangle) + \Theta \tr(P_{k}A)\right) + \rho'(h)\langle \grad h,P_{k}\grad h\rangle \right] \\
&=& \frac{1}{H_k}\left[\rho(h)\left(\mathcal{H}(h)(c_{k}H_{k} - \langle P_{k}\grad h,\grad h\rangle) + c_{k}\Theta H_{k+1}\right) + \rho'(h)\langle \grad h,P_{k}\grad h\rangle \right] \\
&=& \frac{1}{H_k}\left[\rho'(h)c_{k}H_{k} - \rho'(h)\langle P_{k}\grad h,\grad h\rangle + \rho(h)c_{k}\Theta H_{k+1} + \rho'(h)\langle P_{k}\grad h,\grad h\rangle\right] \\
&=& c_{k}\rho(h)\left(\mathcal{H}(h) + \Theta \frac{H_{k+1}}{H_k}\right)
\end{eqnarray*}
\end{proof}

The following result, (Theorem \ref{Thm2}), generalizes  \cite[Thm.10]{alias-impera-rigoli}. We will assume that
\begin{equation}\label{conditionh}\lim_{x\to \infty} \frac{\sigma \circ h (x)}{\varphi (\gamma (x))}=0
\end{equation}where $\sigma $ is given in \eqref{fsigma} and  $\varphi$ is given by \begin{equation}\varphi (t)=\ln \left(\int_{0}^{t}\frac{ds}{\sqrt{G(s)}}+1 \right) \end{equation} while  $(G, \gamma)$ is an Omori-Yau pair in the sense of Definition \ref{def2}.

\begin{theorem}\label{Thm2}
Let $ f\colon  M^n \to I\times_{\rho}P^n = N^{n+1} $ be a properly  immersed hypersurface with  second fundamental form $\alpha$ satisfying \eqref{hipcurvmedia} and $ H_2 > 0 $. Suppose that $P^{n}$ carry an Omori-Yau pair $(G, \gamma)$ for the Hessian, that $ \inf \rho > 0 $  and  $ \mathcal{H} = \displaystyle\frac{\rho'}{\rho} $ satisfies \eqref{hipeta}.  If  $ \tr(\hat{\Pp}_{1})\leq \left(\int_{0}^{\gamma\circ \pi_{\Pp}(f)} \frac{ds}{\sqrt{G(s)}} +1\right)$  and if the height function $ \sigma\circ h $ satisfies \eqref{conditionh}
then
\begin{eqnarray}
\sup_{M}H_{2}^{\frac{1}{2}} \geq \inf_{M}\mathcal{H}(h).
\end{eqnarray}
\end{theorem}
\begin{proof}
By Theorem \ref{corollary2}  the Omori-Yau maximum principle for $\hat{\Ll}_{1}$ holds on $M$ and then  there exists a sequence $ x_j \in M $ such that
\begin{eqnarray}\label{eqThm2}
\frac{1}{j} &>& \hat{\Ll}_{1}\sigma\circ h(x_j) = n(n-1)\rho(h(x_j))\left(\mathcal{H}(h(x_j)) + \Theta(x_j)\frac{H_{2}}{H_1}(x_j)\right) \nonumber \\
&\geq & n(n-1)\rho(h(x_j))\left(\mathcal{H}(h(x_j)) - \frac{H_{2}}{H_1}(x_j)\right) \\
&\geq & n(n-1)\rho(h(x_j))\left(\mathcal{H}(h(x_j)) - \sqrt{H_{2}}(x_j)\right) \nonumber \\
&\geq & n(n-1)\rho(h(x_j))\left(\inf_{M}\mathcal{H}(h) - \sup_{M}\sqrt{H_2}\right). \nonumber
\end{eqnarray}Where we used Lemma \ref{lemma2} in the right hand side the first line of \eqref{eqThm2}.
Since $ \sigma $ is strictly increasing and $ h(x_j) \to \sup h $ as $ j \to +\infty $, we obtain that
\begin{eqnarray*}
\sup_{M}\sqrt{H_2} \geq \inf_{M}\mathcal{H}(h).
\end{eqnarray*}
\end{proof}

\begin{corollary}\label{corollary3}
Let $ P^n $ be a complete, non-compact Riemannian manifold whose radial sectional curvature satisfies
\begin{eqnarray}
K_{P}^{rad} \geq - C\cdot  G(r) ,
\end{eqnarray}
where $G\in C^{\infty}([0,+\infty)) $ is even at the origin and satisfies iv) in Theorem \ref{Thm1}, $r(x)={\rm dist}_{M}(x_{0}, x)$. If $ f \colon M^n \to I\times_{\rho}P^n = N^{n+1} $ is a properly immersed hypersurface with $ H_2 > 0 $, satisfying \eqref{hipeta}, \eqref{hipcurvmedia}, \eqref{conditionh} and $ \inf \rho > 0 $ . If   $ \tr(\hat{\Pp}_{1})\leq \left(\int_{0}^{r\circ \pi_{\Pp}(f)} \frac{ds}{\sqrt{G(s)}} +1\right)$
then
\begin{eqnarray}
\sup_{M}H_{2}^{\frac{1}{2}} \geq \inf_{M}\mathcal{H}(h).
\end{eqnarray}
\end{corollary}

Our next result  is just a version of Theorem \ref{Thm2} for higher order mean curvatures.
\begin{theorem}\label{Thm3}
Let $ f\colon  M^n \to I\times_{\rho}P^n = N^{n+1} $ be a proper isometric immersion  with an elliptic point and $ H_k > 0 $. Suppose that the second fundamental form $\alpha$ satisfies \eqref{hipcurvmedia} and that $ \inf \rho > 0 $. Moreover, assume that  $\displaystyle \mathcal{H} = \rho'/\rho $ satisfies \eqref{hipeta} and that $P^{n}$ carry an Omori-Yau pair $(G, \gamma)$ for the Hessian.  If  $ \tr(\hat{\Pp}_{k-1})\leq \left(\int_{0}^{\gamma\circ \pi_{\Pp}(f)} \frac{ds}{\sqrt{G(s)}} +1\right)$, $3\leq k\leq n$ and if the height function $ \sigma\circ h $ satisfies \eqref{conditionh}
then
\begin{eqnarray}
\sup_{M}H_{k}^{\frac{1}{k}} \geq \inf_{M}\mathcal{H}(h).
\end{eqnarray}
\end{theorem}
Likewise, we have the corollary
\begin{corollary}\label{corollary4}
Let $ P^n $ be a complete, non-compact Riemannian manifold whose radial sectional curvature satisfies
\begin{eqnarray}
K_{P}^{rad} \geq - C\cdot  G(r)
\end{eqnarray}
If $ f \colon M^n \to I\times_{\rho}P^n = N^{n+1} $ is a properly immersed hypersurface with $ H_k > 0 $, satisfying \eqref{hipeta}, \eqref{hipcurvmedia},  \eqref{conditionh} and $ \inf \rho > 0 $. If    $ \tr(\hat{\Pp}_{1})\leq \left(\int_{0}^{r\circ \pi_{\Pp}(f)} \frac{ds}{\sqrt{G(s)}} +1\right)$
then
\begin{eqnarray}
\sup_{M}H_{k}^{\frac{1}{k}} \geq \inf_{M}\mathcal{H}(h).
\end{eqnarray}
\end{corollary}
\section{Slice Theorem}\label{sec3}In this section we  will prove an extension of  the Slice Theorem  proved by L. Alias, D. Impera and M. Rigoli  \cite[Thms. 16 \& 21]{alias-impera-rigoli}. We start with the following lemma.
\begin{lemma}\label{thetalema}
Let $ f \colon M^n \to I\times_{\rho}P^n = N^{n+1} $ be a hypersurface with non-vanishing mean curvature and suppose that the height function $h$ satisfies
\begin{eqnarray}\label{cond.h}
\lim_{x\to \infty}\frac{h(x)}{\varphi(\gamma(x))} = 0,
\end{eqnarray}
where  $ \varphi $ is given  in \eqref{defvarphi} and $\gamma$ in the statement of Theorem \ref{Thm1}. Assume that $ \mathcal{H}' \geq 0 $ and that the angle function $ \Theta $ does not change sign. Choose on $ M^n $ the orientation so that $ H_1 > 0 $. Suppose the Omori-Yau maximum principle for the Laplacian holds on $ M^n $. Then we have that
\begin{enumerate}
\item[i)] If $ \Theta \leq 0 $ then $ \mathcal{H}(h) \geq 0 $,
\item[]
\item[ii)] If $ \Theta \geq 0 $ then $ \mathcal{H}(h) \leq 0 $.
\end{enumerate}
\end{lemma}
\begin{proof}
By hypothesis, we have that Omori-Yau maximum principle for the Laplacian holds for the height function $ h $, therefore,  there exists a sequences $ \{x_j\} , \{y_j\} \subset M^n $ such that
\begin{eqnarray}\label{a1}
\lim_{j\to +\infty}h(x_j) = \sup h = h^{*} , \ \ \vert \grad h\vert^{2}(x_j) < \left(\frac{1}{j}\right)^2 \ \ \mbox{and} \ \ \Delta h(x_j) < \frac{1}{j}
\end{eqnarray}
and
\begin{eqnarray}\label{a2}
\lim_{j\to +\infty}h(y_j) = \inf h = h_{*} , \ \ \vert \grad h\vert^{2}(y_j) < \left(\frac{1}{j}\right)^2 \ \ \mbox{and} \ \ \Delta h(x_j) > -\frac{1}{j}.
\end{eqnarray}
On the other hand, we know that $ \Delta h = \mathcal{H}(h)(n-\vert \grad h\vert^{2}) + nH_{1}\Theta $. Therefore, supposing that $ \Theta \geq 0 $ we get by \eqref{a1} that
\begin{eqnarray*}
0 \geq -nH_{1}(x_j)\Theta(x_j) > -\frac{1}{j} + \mathcal{H}(h(x_j))(n-\vert \grad h\vert^{2}(x_j))
\end{eqnarray*}
and then, for $ j \gg 1 $ we obtain
\begin{eqnarray*}
0 \geq \frac{-nH_{1}(x_j)\Theta(x_j)}{n-\vert \grad h\vert^{2}(x_j)} > -\frac{1}{j(n-\vert \grad h\vert^{2}(x_j))} + \mathcal{H}(h(x_j)).
\end{eqnarray*}
letting $ j \to +\infty $, we have
\begin{eqnarray}
0 \geq \limsup_{j\to +\infty}  \mathcal{H}(h(x_j)) = \mathcal{H}(h^{*}) \geq \mathcal{H}(h).
\end{eqnarray}
Similar proof gives  the item i).
\end{proof}

Define the operator $\mathcal{L}_{1} = \tr(\mathcal{P}_{1}\circ \hess)= (n-1)\mathcal{H}(h)\Delta - \Theta \Ll_{1}  $ where
$
\mathcal{P}_{1} = (n-1)\mathcal{H}(h)I - \Theta P_1 .
$ Using that
$$\Delta h = \mathcal{H}(h)(n-\vert \grad h\vert^{2}) + nH_{1}\Theta $$ and  $$L_{1}(h)= n(n-1)\left(\mathcal{H}(h)H_1 + \theta H_{2}\right) - \mathcal{H}(h)\langle P_1 \grad h, \grad h\rangle $$ we obtain
$$\begin{array}{rll}
\mathcal{L}_{1}(h)& = &n(n-1)\left(\mathcal{H}^{2}(h) - \Theta^{2}H_{2}\right) - (n-1)\mathcal{H}(h)\langle \mathcal{P}_{1}\grad h,\grad h \rangle \end{array}
$$
$$\begin{array}{rll}
\mathcal{L}_{1}(\sigma \circ h) &= &n(n-1)\rho(h)\left(\mathcal{H}^{2}(h) - \Theta^{2}H_{2}\right)
\end{array}
$$
The following theorem extends \cite[Thm. 16]{alias-impera-rigoli} which extends \cite[Thm 2.9]{alias-dajczer-1}.

\begin{theorem}\label{thm5}
Let $ f : M^n \to I\times_{\rho}P^n = N^{n+1} $ be a complete hypersurface of constant positive 2-mean curvature $ H_2 >0$ with radial sectional curvature $K^{rad}_{M}$ satisfying
\begin{eqnarray}\label{sectionalcurvature2}
K_{M}^{rad} \geq - B^{2}\prod_{j=1}^{\ell}\left[\ln^{(j)}\left(\int_{0}^{r}\frac{ds}{\sqrt{G(s)}} + 1\right)+1\right]^{2}G(r),\,\,\,\,{\rm for} \,\,\,\, r(x)\gg 1,
\end{eqnarray}
 where $G\in C^{\infty}([0,+\infty)) $ is even at the origin and satisfies iv) in Theorem \ref{Thm1}, $r(x)={\rm dist}_{M}(x_{0}, x)$ and  $B\in \mathbb{R}$. Suppose also that height function satisfies the conditions  \eqref{conditionh}, \eqref{cond.h} and that
\begin{eqnarray}\label{cdh1}
\vert H_1 \vert(r) \leq \frac{1}{n(n-1)}\left(\int_{0}^{r}\frac{ds}{\sqrt{G(s)}} + 1\right),
\end{eqnarray}
and
\begin{eqnarray}\label{cdmh}
\vert \mathcal{H}\vert(t) \leq \frac{1}{n(n-1)}\left(\int_{0}^{t}\frac{ds}{\sqrt{G(s)}} + 1\right)
\end{eqnarray}
If $ \mathcal{H}^{'} > 0 $ almost everywhere and the angle function $ \Theta $ does not change sign, then $ f(M^{n}) $ is a slice.
\end{theorem}
\begin{proof}
Taking an orientation on $ M^n $ in which $ H_1 > 0 $ we have, by Lemma \ref{thetalema}, that if $ \Theta \leq 0 $, then $ \mathcal{H}(h) \geq $ 0 and therefore the operator $ \mathcal{P}_1 $ is positive semi-definite. Therefore, which implies that $ \mathcal {L}_1 $ is semi-elliptic. Furthermore, by \eqref{cdh1} and \eqref{cdmh}, we have
\begin{eqnarray*}
\tr \mathcal{P}_1 = n(n-1)\left(\mathcal{H}(h) - H_{1}\Theta\right)
&\leq & \left(\int_{0}^{\gamma}\frac{ds}{\sqrt{G(s)}} + 1\right)
\end{eqnarray*}
By Corollary \ref{corollary1} the  Omori-Yau maximum principle for the operator $ \mathcal{L}_1 $   holds on $ M^n $  with the functions $ h $ and $ \sigma \circ h $ (conditions \eqref{conditionh} and \eqref{cond.h}). Thus, there is a sequence $ \{x_{j}\} \subset M^n $ such that
\begin{enumerate}
\item[i.] $\lim_{j\to +\infty} \sigma(h(x_j)) = (\sigma \circ h)^{*}$
\item[ii.]$\vert \grad(\sigma \circ h)\vert(x_j) = \rho(h(x_j))\vert \grad h\vert(x_j) < \frac{1}{j} $
\item[iii.]$\mathcal{L}_{1}(\sigma \circ h)(x_j) < \frac{1}{j}$.
\end{enumerate}
We know that $ \mathcal{H}(h) \geq $ 0, and hence $ \rho'(h) \geq 0 $ and $ \rho $ is increasing. Since $ \sigma $ is strictly increasing, we have $ \displaystyle{\lim_{j \to + \infty}h(x_j) = h^{*}\leq \infty} $. Thus, $ \displaystyle{\lim_{j \to + \infty}\rho(h(x_j)) > c > 0} $. Since $ \Theta^{2} \leq 1 $, the item iii. tells us that
\begin{eqnarray*}
\frac{1}{j} &>&= n(n-1)\rho(h(x_j))\left(\mathcal{H}^{2}(h(x_j)) - \Theta^{2}(x_j)H_{2}\right) \\
&>& n(n-1)\rho(h(x_j))\left(\mathcal{H}^{2}(h(x_j)) - H_{2}\right).
\end{eqnarray*}
Making $ j \to +\infty $, we gets
\begin{eqnarray}\label{eq.h2-1}
\mathcal{H}^{2}(h^{*})\leq H_2 .
\end{eqnarray}

On the other hand, there is a sequence $ \{y_{j}\} \subset M^n $, in such a way that
\begin{enumerate}
\item[i.] $\lim_{j\to +\infty} h(y_j) = h_{*} $
\item[ii.]$\vert \grad h\vert(y_j) = \vert \grad h\vert(y_j) < \frac{1}{j} $
\item[iii.]$\mathcal{L}_{1}(h)(y_j) > -\frac{1}{j}$.
\end{enumerate}
Since $ \mathcal{P}_1 $ is positive semi-definite, there exists a constant $ \beta \geq 0 $ such that $$ \langle \mathcal{P}_{1}\grad h,\grad h \rangle \geq \beta\vert \grad h\vert^{2} .$$ Thus,
\begin{eqnarray*}
 -\frac{1}{j} &<& n(n-1)\left(\mathcal{H}^{2}(h(y_j)) - \Theta^{2}(y_j)H_{2}(y_j)\right) - \mathcal{H}(h(y_j))\langle \mathcal{P}_{1}\grad h(y_j),\grad h(y_j) \rangle \\ && \\
&\leq & n(n-1)\left(\mathcal{H}^{2}(h(y_j)) - \Theta^{2}(y_j)H_{2}(y_j)\right) - \beta \vert \grad h\vert^{2}(y_j) \\&& \\
&\leq & n(n-1)\left(\mathcal{H}^{2}(h(y_j)) - \Theta^{2}(y_j)H_{2}(y_j)\right).
\end{eqnarray*}
Since $ 1 \geq \Theta^{2}(y_j) = 1 - \vert \grad h\vert^{2}(y_j) > 1 - \frac{1}{j} $, by the item ii. we have $ \lim_{j \to +\infty}\Theta^{2}(y_j) = 1 $.
Doing $ j\to +\infty $, we get
\begin{eqnarray}\label{eq.h2-2}
H_2 \leq \mathcal{H}^{2}(h_{*}).
\end{eqnarray}
Combining \eqref{eq.h2-1} with \eqref{eq.h2-2}, we get that
\begin{eqnarray*}
\mathcal{H}(h^{*}) \leq H_2 \leq \mathcal{H}(h_{*})
\end{eqnarray*}
and as $ \mathcal{H} $ is an increasing function, conclude that $ h^{*} = h_{*}< \infty $.

If $ \Theta \geq 0 $, we applied in a manner entirely analogous the Omori-Yau maximum principle to the  semi-elliptic operator $ - \mathcal{L}_{1} $.
\end{proof}

To extend the previous findings in the case of higher order curvatures, let us define for each $ 2\leq k \leq n $, the operators

\begin{eqnarray}
\mathcal{L}_{k-1} = \tr\left(\mathcal{P}_{k-1}\circ \hess\right),
\end{eqnarray}
where
\begin{eqnarray}
\mathcal{P}_{k-1} = \sum_{j=0}^{k-1}(-1)^{j}\frac{c_{k-1}}{c_j}\mathcal{H}^{k-1-j}(h)\Theta^{j}P_{j} .
\end{eqnarray}

Observe that
\begin{eqnarray*}
\mathcal{P}_{k-1} = \frac{c_{k-1}}{c_{k-2}}\mathcal{H}(h)\mathcal{P}_{k-2} + (-1)^{k-1}\Theta^{k-1}P_{k-1}
\end{eqnarray*}
and consequently
\begin{eqnarray}
\mathcal{L}_{k-1} = \frac{c_{k-1}}{c_{k-2}}\mathcal{H}(h)\mathcal{L}_{k-2} + (-1)^{k-1}\Theta^{k-1}L_{k-1} .
\end{eqnarray}

By induction, we see that
\begin{eqnarray}\label{exp.lkh}
\mathcal{L}_{k-1}h = c_{k-1}\left(\mathcal{H}^{k}(h) - (-1)^{k}\Theta^{k}H_{k}\right) - \mathcal{H}(h)\langle \mathcal{P}_{k-1}\grad h,\grad h\rangle
\end{eqnarray}
and
\begin{eqnarray}\label{exp.lksigmah}
\mathcal{L}_{k-1}\sigma(h) = c_{k-1}\rho(h)\left(\mathcal{H}^{k}(h) - (-1)^{k}\Theta^{k}H_{k}\right).
\end{eqnarray}
In fact, assuming that the expression \eqref{exp.lkh} is valid for $ k-2 $, we have:
\begin{eqnarray*}
\mathcal{L}_{k-1}h &=& \frac{c_{k-1}}{c_{k-2}}\mathcal{H}(h)\mathcal{L}_{k-2}h + (-1)^{k-1}\Theta^{k-1}L_{k-1}h \\
&=& \frac{c_{k-1}}{c_{k-2}}\mathcal{H}(h)\left[c_{k-2}\left(\mathcal{H}^{k-1}(h) - (-1)^{k-1}\Theta^{k-1}H_{k-1}\right) - \mathcal{H}(h)\langle \mathcal{P}_{k-2}\grad h,\grad h\rangle \right] + \\
&& + (-1)^{k-1}\Theta^{k-1}\left[\mathcal{H}(h)\left(c_{k-1}H_{k-1} - \langle P_{k-1}\grad h,\grad h\rangle\right) + c_{k-1}\Theta H_{k} \right] \\
&=& c_{k-1}\mathcal{H}^{k}(h) - (-1)^{k-1}c_{k-1}\mathcal{H}(h)\Theta^{k-1}H_{k-1} - \frac{c_{k-1}}{c_{k-2}}\mathcal{H}^{2}(h)\langle \mathcal{P}_{k-2}\grad h,\grad h\rangle + \\
&& + (-1)^{k-1}c_{k-1}\mathcal{H}(h)\Theta^{k-1}H_{k-1} + (-1)^{k-1}c_{k-1}\Theta^{k}H_{k} - \\
&&  - (-1)^{k-1}\mathcal{H}(h)\Theta^{k-1}\langle P_{k-1}\grad h,\grad h\rangle \\
&=& c_{k-1}\left(\mathcal{H}^{k}(h) - (-1)^{k}\Theta^{k}H_{k}\right) - \mathcal{H}(h)\langle \mathcal{P}_{k-1}\grad h,\grad h\rangle .
\end{eqnarray*}

The proof of the expression \eqref{exp.lksigmah} follows similarly. \\
The next theorem extends the Theorem \ref{thm5} for the case of higher order curvatures and your proof is analogous and use only the expressions \eqref{exp.lkh} and \eqref{exp.lksigmah}.

\begin{theorem}\label{thm6}
Let $ f: M^n \to I\times_{\rho}P^n = N^{n+1} $ be a complete hypersurface of constant positive k-mean curvature $ H_k > 0 $, $ 3\leq k \leq n$ with radial sectional curvature $K^{rad}_{M}$ satisfying
\begin{eqnarray}\label{sectionalcurvature3}
K_{M}^{rad} \geq - B^{2}\prod_{j=1}^{\ell}\left[\ln^{(j)}\left(\int_{0}^{r}\frac{ds}{\sqrt{G(s)}} + 1\right)+1\right]^{2}G(r),\,\,\,\,{\rm for} \,\,\,\, r(x)\gg 1,
\end{eqnarray}
 where $G\in C^{\infty}([0,+\infty)) $ is even at the origin and satisfies iv) in Theorem \ref{Thm1}, $r(x)={\rm dist}_{M}(x_{0}, x)$ and  $B\in \mathbb{R}$. Suppose also that height function satisfies the conditions  \eqref{conditionh}, \eqref{cond.h} and that
\begin{eqnarray}\label{cdh1}
\vert H_{1}^{j}\vert \cdot \vert \mathcal{H}^{k-1-j}(h)\vert(r) \leq \frac{1}{kc_{k-1}}\left(\int_{0}^{r}\frac{ds}{\sqrt{G(s)}} + 1\right), \ \ \forall j = 0,...,k-1.
\end{eqnarray}
Assume that there exists an elliptic point in $ M^n$. If $ \mathcal{H}^{'} > 0 $ almost everywhere and the angle function $ \Theta $ does not change sign, then $ f(M^{n}) $ is a slice.
\end{theorem}

\end{document}